\documentclass[a4paper]{amsart}
\usepackage{amsmath,amssymb,braket,amsthm}
\usepackage{enumitem}
\usepackage[dvips]{graphicx}

\newcommand{\NN}{\mathbb{N}}
\newcommand{\RR}{\mathbb{R}}
\newcommand{\ZZ}{\mathbb{Z}}
\newcommand{\AAA}{\mathcal{A}}
\newcommand{\BBB}{\mathcal{B}}

\newcommand{\PPP}{P}

\newcommand{\partitionof}{\vdash}

\newcommand{\numof}[1]{\# #1}

\newcommand{\Stab}{\operatorname{Stab}}
\newcommand{\codim}{\operatorname{codim}}
\newcommand{\ideal}{\mathcal{I}}

\newcommand{\nt}{\operatorname{\kappa}}

\theoremstyle{plain}   
\newtheorem{thm}{Theorem}[section]
\newtheorem{theorem}[thm]{Theorem}

\newtheorem{lemma}[thm]{Lemma}
\newtheorem{cor}[thm]{Corollary}
\newtheorem{proposition}[thm]{Proposition}

\theoremstyle{remark}  
\newtheorem{remark}[thm]{Remark}
\newtheorem{rem}[thm]{Remark}
\newtheorem{example}[thm]{Example}
\newtheorem*{acknowledgments}{Acknowledgments}

\theoremstyle{definition}  
\newtheorem{definition}[thm]{Definition}

\allowdisplaybreaks[3]

\usepackage{natbib}

\begin{document}

\title[hyperplane arrangements generated by generic points]%
{On intersection lattices of hyperplane arrangements generated by generic points}

\author[Koizumi]{Hiroshi KOIZUMI}
\address[H. Koizumi]{Mathematical Informatics, 
Graduate School of Information Science and Technology, 
University of Tokyo.}

\author[Numata]{Yasuhide NUMATA}
\address[Y. Numata]{Mathematical Informatics, 
Graduate School of Information Science and Technology, 
University of Tokyo.}

\author[Takemura]{Akimichi TAKEMURA}
\address[A. Takemura]{Mathematical Informatics, 
Graduate School of Information Science and Technology, 
University of Tokyo.}

\address[Y. Numata and A. Takemura]{JST CREST}

\thanks{The second and the third authors are supported by JST CREST}



\begin{abstract}
We consider hyperplane arrangements generated by generic points
and study their intersection lattices.  
These  arrangements are known to be
equivalent to discriminantal arrangements. 
We show a fundamental structure of the intersection lattices by 
decomposing the poset ideals 
as  direct products of smaller lattices corresponding to smaller dimensions.
Based on this decomposition we compute the M\"obius functions of the lattices 
and  the characteristic polynomials of the arrangements up to dimension six.
\end{abstract}

\maketitle
\section{Introduction}
\label{sec:intro}
Consider a set of $n$ $( >  d)$ generic points  $\PPP=\Set{p_1,\ldots,p_n}$ 
in a $d$-dimensional vector space $V=K^d$ over a field $K$ of
characteristic zero.    
For $X\subset \PPP$ 
let $H_{X}$ denote 
the affine hull of $X$.
Let 
\[
\AAA=\Set{H_{X}| X\subset \PPP, \numof{X} = d}
\]
be the set of all hyperplanes defined by $H_{X}$ for some $X\subset\PPP$, 
$\numof{X}=d$.
Here we assume that points $p_1,\ldots,p_n$ are generic in the
sense of \cite{athanasiadis2000}.  
Then combinatorial properties of the arrangement  $\AAA$ 
does not depend on the points. 
Since in this paper
we are interested only  in the combinatorial properties of $\AAA$, 
we denote the arrangement by $\AAA_{n,d}$. 
We decompose the poset ideals  of the intersection lattice of $\AAA_{n,d}$
into direct products of smaller lattices corresponding to smaller dimensions.
Based on this decomposition 
we give an explicit description of the M\"obius functions 
and the characteristic polynomials of the intersection lattices 
for $d\le 6$ and for all $n>d$.

By Theorem 2.2 of \cite{falk1994}, 
$\AAA_{n,d}$ is 
equivalent to the discriminantal arrangement $\BBB(n,n-d-1)$ 
of \cite{manin-schechtman1989}.
Relevant facts on the discriminantal arrangement 
are given in Section 5.6 of \cite{orlik-terao-book},
\cite{bayer-brandt1997} and \cite{athanasiadis2000}.
We prefer to work with $\AAA_{n,d}$ 
because we utilize the recursive structure of $\AAA_{n,d}$ 
with respect to $d$.

The organization of this paper is the following.
In Section \ref{sec:def} we set up our definition and notation.  
In particular
following \cite{athanasiadis2000} 
we interpret the intersection lattice of our arrangement
in set theoretical terminology. 
We also give illustrations for $d\le 3$.
In Section \ref{sec:main}, 
we show the fundamental structure of the intersection lattice of $\AAA_{n,d}$,
which is the main result of this paper.
Based on the main result, 
in Section \ref{sec:computation}
we compute 
the M\"obius function of the intersection lattice,
the number of elements of a particular type of the intersection lattice,
and the characteristic polynomials of the arrangements up to $d=6$
and for all $n>d$.

\begin{acknowledgments}
The authors are very grateful to 
Hidehiko Kamiya and Hiroaki Terao for very useful comments.
\end{acknowledgments}

\section{Definition and Notation}
\label{sec:def}

We denote the intersection lattice of $\AAA_{n,d}$ by
\begin{equation*}
L(\AAA_{n,d})=\set{H_1 \cap \cdots \cap H_k | H_1,\ldots, H_k \in \AAA_{n,d}},
\end{equation*}
where the sets are ordered by reverse inclusion.  
Contrary to the usual convention, 
here we consider that 
$\emptyset = \bigcap_{X\colon \numof{X}=d} H_X$ belongs to $L(\AAA_{n,d})$,
so that $L(\AAA_{n,d})$ is not only a poset but also a lattice 
(cf.\ Proposition 2.3 of \cite{stanley-introduction}).
In usual convention, 
this corresponds to the coning $c\AAA_{n,d}$ of $\AAA_{n,d}$, 
except that we do not add a coordinate hyperplane.  
The reason for this unconventional definition
is that $\emptyset \in L(\AAA_{n,d})$ plays an essential role
for recursive description of $L(\AAA_{n,d})$.

We now follow \cite{athanasiadis2000} 
to give an interpretation of $L(\AAA_{n,d})$ in set theoretical terminology.

\begin{definition}
For a finite set  $X$, 
we define 
\[
\codim_d(X) = d+1-\numof{X}.
\]
For distinct finite sets $T_1,\ldots, T_l$,
we define 
\begin{align}
\nonumber
\rho_d(\Set{T_1,\ldots,T_l})
&= \codim_d T_1 +\cdots+\codim_d T_l,\\
D_d(\Set{T_1,\ldots,T_l})
&=\codim_d(T_1\cap\cdots\cap T_l)-\rho_d(\Set{T_1,\ldots,T_l}).\nonumber
\end{align}
We also define $\rho_d(\emptyset)=\rho_d(\Set{ }) = 0$.
\end{definition}

\begin{remark}
\label{remark:D for 2}
By definition, it follows that
\begin{align}
\label{eq:Dd-alternative}
D_d&(\Set{T_1,\ldots,T_l})
=-(l-1)(d+1)+\numof{T_1}+\cdots+\numof{T_l}-\numof{(T_1\cap\cdots \cap T_l)}. 
\end{align}
In particular for $T_1 \neq T_2$, 
\begin{equation}
\label{eq:D-for-2}
D_d(\Set{T_1,T_2})=\numof{(T_1\cup T_2)}-(d+1).
\end{equation}
\end{remark}

\begin{remark}
\label{remark:nonempty-intersecton}
For $Y\subset X$, $\codim_{d-\numof{Y}}(X\setminus Y)=\codim_d(X)$.
This implies the following fact. 
Let  $U \subset T_1 \cap \dots \cap T_l$. Then
\begin{align*}
\rho_{d-\numof{U}}(\set{ T_1\setminus U, \dots, T_l \setminus U})&=
\rho_{d}(\set{ T_1, \dots, T_l}), \\
D_{d-\numof{U}}(\set{ T_1\setminus U, \dots, T_l \setminus U})&=
D_{d}(\set{ T_1, \dots, T_l}).
\end{align*}
\end{remark}

\begin{definition}
\label{def:DefOfLattice}
For $d>0$ and $n>d$, we define $L(n,d)$ to be the set of
$T\subset 2^{\Set{1,\ldots,n }}$ satisfying the following two conditions:
\begin{enumerate}[label=\arabic{*}) ]
\item \label{item:DefOfLattice1}
$D_d(T')>0$ for all $T'\subset T$ with $\numof{T'}>1$.
\item \label{item:DefOfLattice2}
$0\leq \numof{T_i} \leq d$ for all $T_i \in T$.
\end{enumerate}
Moreover we define the partial ordering $ < $ on $L(n,d)$ by
\begin{align}
T<T' 
&\iff 
\begin{cases}
\rho_d(T) < \rho_d(T') & \text{and} \\
\text{$\forall T_i \in T$, $\exists T_j' \in T' $ such that $T_j' \subset T_i$.}
\end{cases}
\label{eq:partial-order}
\end{align} 
\end{definition}

Let $\PPP=\Set{p_1,\ldots,p_{n}}$ 
be a collection of generic points in $V$
in the sense of Section \ref{sec:intro}.
For $X \subset \Set{1,\ldots,n}$,
$0\leq \numof{X} \leq d$, 
define $H_X$ to be the affine hull $H_{\Set{p_i|i\in X}}$.
Since $n>d$, there exists a subset $X'\subset \Set{1,\ldots,n}$ such that 
$X' \cap X =\emptyset$ and $\numof{(X\cup X')}=d+1$.
Hence
\begin{align*}
H_X&=H_{\Set{p_{i_1},\ldots,p_{i_l}}}\\
&=
\bigcap_{k\in X'}H_{\Set{p_i | i \in X \cup X'} \setminus \Set{ p_{k} } } 
\in L(\AAA_{n,d}).
\end{align*}
This mapping induces
a map from $L(n,d)$ to $L(\AAA_{n,d})$,
or equivalently, $T\in L(n,d)$ corresponds to 
$H(T)=\bigcap_{T_i\in T} H_{T_i} \in L(\AAA_{n,d})$.
By this correspondence,  $L(n,d)$ is isomorphic to $L(\AAA_{n,d})$ 
as lattices (\cite{athanasiadis2000}, \cite{falk1994}).

\begin{remark}
\label{rem:max-element}
$L(n,d)$ is a graded poset with the rank function $\rho_d$.
$\emptyset=\set{}$ is the minimum element of $L(n,d)$ with $\rho_d(\emptyset)=0$
and $\set{\emptyset}$ ($\emptyset\subset \set{1,\dots,n}$) is the maximum element of $L(n,d)$ 
with $\rho_d(\set{\emptyset})=d+1$.
In the one-to-one correspondence between $L(n,d)$ and $L(\AAA_{n,d})$,
$H(\emptyset)=V=K^d$ and $H(\set{\emptyset})=\emptyset$ $(\subset K^d)$.  
In the case $d=0$,
the condition \ref{item:DefOfLattice2}
in Definition \ref{def:DefOfLattice}
implies $\numof{T_i} = 0$ for $T_i \in T \in L(n,0)$.
Hence $L(n,0)$ is the poset of two elements
\begin{equation*}
 L(n,0)=\set{\emptyset, \set{\emptyset}}
 \end{equation*}
independent of $n$.
\end{remark}

Let $d$ be a nonnegative integer.
We call 
a weakly-decreasing sequence $\delta=(\delta_1,\delta_2,\ldots)$
of nonnegative integers such that $\sum_{i} \delta_i=d$
a {\em partition} of $d$.
We write $\delta\partitionof d$ to say that
$\delta$ is a partition of $d$.
We also regard a partition 
as a multiset of positive integers. 
For example, $\Set{ \delta \partitionof 3 }=\set{(3),(2,1),(1,1,1)}$,
and $\Set{ \delta \partitionof 0 }$ 
is the set consisting of  the unique partition of zero, 
which is denoted by $(0)$.

\begin{definition}
Let  $T=\set{T_1, \dots, T_l}\in L(n,d)$.
Without loss of generality
assume  $\numof{T_1}\leq \dots \leq \numof{T_l}$.  
We call 
\[
\gamma_d(T)=(\codim_d(T_1),\dots,\codim_d(T_l)) \partitionof \rho_d(T)
\]
the {\em type} of $T$.  
\end{definition}
\begin{example}
For any $d$, $\gamma_d(\emptyset)=(0)$ and $\gamma_d(\set{\emptyset})=(d+1)$.
\end{example}

\begin{definition}
For $T\in L(n,d)$,
we define $\ideal_{n,d}(T)$ to be the poset ideal generated by $T$,
i.e., $\ideal_{n,d}(T)=\Set{S \in L(n,d)|S\leq T}$.
\end{definition}

Finally we define 
the M\"obius function $\mu_{n,d}$ of the poset $L(n,d)$, 
which will be studied in Section \ref{sec:computation}.  
Define $\mu_{n,d}$ by
\[
\mu_{n,d}(T,T)=1, \quad 
\sum_{S\colon T\leq S\leq T'} \mu_{n,d}(T,S)=0, \ T < T' .
\]
We write  $\mu_{n,d}(T) = \mu_{n,d}(\emptyset,T)$.
The characteristic polynomial $\chi_{n,d}(t)$ of the poset $L(n,d)$ 
(cf.\ Section 3.10 of \cite{stanley-book-1}) is defined by
\begin{equation}
\chi_{n,d}(t) =\sum_{T\in L(n,d)} \mu_{n,d}(T) t^{d+1-\rho_d(T)}.
\label{eq:characteristic-function}
\end{equation}

Note that 
the usual characteristic polynomial $\chi(\AAA_{n,d},t)$
of the non-central arrangement $\AAA_{n,d}$ 
is given as
\[
\chi(\AAA_{n,d},t)
=\sum_{T\in L(n,d), \, T\neq \set{\emptyset}} \mu_{n,d}(T) t^{d-\rho_d(T)}
=\frac{\chi_{n,d}(t)- \mu_{n,d}(\set{\emptyset})}{t}.
\]
Conversely from $\chi(\AAA_{n,d},t)$ 
we can evaluate  $\mu_{n,d}(\set{\emptyset})=-\chi(\AAA_{n,d},1)$
since  $\chi_{n,d}(1)=0$. 
Equivalently
\begin{equation}
\label{eq:mumaxphi}
  \mu_{n,d}(\set{\emptyset})
=- \sum_{T\in L(n,d), \, T\neq \set{\emptyset}} \mu_{n,d}(T).
\end{equation}

\subsection{Illustration of the posets up to  dimension three}
\label{subsec:illustration}

We illustrate the above definitions with $d=0,\dots,3$.
For $d=0$ we already saw $L(n,0)=\set{\emptyset,\set{\emptyset}}$.
In particular $\mu_{n,0}(\set{\emptyset})=-1$.

Let $d=1$. 
In $L(n,1)$, 
in addition to the minimum $\emptyset$ and the maximum $\set{\emptyset}$, 
there are $n$ rank one elements $\Set{\set{i}}$, $i=1,\dots,n$,  
with $\mu_{n,1}(\Set{\set{i}})=-1$. 
Hence $\chi(\AAA_{n,1},t)=t-n$.
The value $\mu_{n,1}(\set{\emptyset})=n-1$ is relevant for 
$d > 1$.  

\begin{figure}
\begin{center}
\includegraphics[width=7cm]{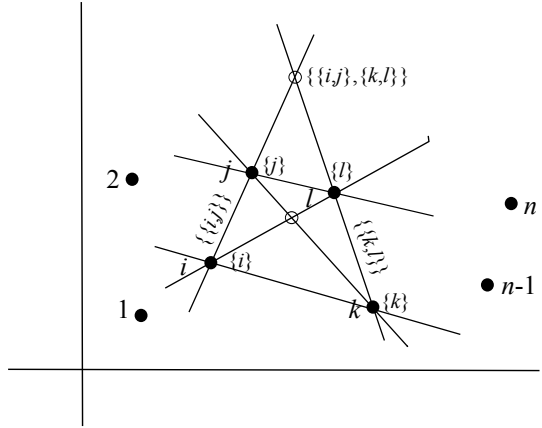}
\end{center}
\caption{Arrangement for dimension two}
\label{fig:d2}
\end{figure}

The case  $d=2$ is already discussed 
in Section 7 of \cite{manin-schechtman1989} 
and Section 5.6 of \cite{orlik-terao-book}.
However we present it here from our viewpoint.
As shown in Figure \ref{fig:d2},
each line (rank one element) is labeled by a pair of points, such as $T=\set{\set{i,j}}$,
which is a line connecting points $p_i$ and $p_j$.  
There are two types  of points (rank two elements).  
The first type is an element of type $(2)\partitionof 2$.
Each element  $\set{\set{i}}$ of type $(2)\partitionof 2$  
corresponds to an original point in $\PPP$.  
The second type is an element of type $(1,1)\partitionof 2$.
Each element $T=\set{\set{i,j},\set{k,l}}$  of type $(1,1)$ corresponds
the intersection of two lines, 
depicted by a white circle in Figure \ref{fig:d2}.  
The M\"obius function is evaluated as 
$\mu_{n,2}(\set{\set{i}})=n-2$ and  
$\mu_{n,2}(\set{\set{i,j},\set{k,l}})=1$.

\begin{rem}
\label{rem:n-existence}
In this paper 
we are assuming that $n > d$ 
so that $\AAA_{n,d}$ is a non-central arrangement.
We usually think of $n$ as ``sufficiently large'' compared to $d$.  
Relevant quantities are polynomials in $n$ and 
these polynomials are determined by sufficiently large $n$.  
However our polynomials hold for all $n>d$ with appropriate qualifications.  
For example, 
the second type $\set{\set{i,j},\set{k,l}}$ of $L(n,2)$ exists 
if and only if $n\ge 4$.  
As long as $n\ge 4$, $\mu_{n,2}(\set{\set{i,j},\set{k,l}})=1$.
In general, 
when we write  $T\in L(n,d)$, 
this $T$ has to exist in $L(n,d)$.
Actually we are interested in the existence of some $T'$
with the same type as $T$, i.e. $\gamma_d(T')=\gamma_d(T)$.
The existence implies 
that $n$ has to be larger than or equal to some specific value,  
say $n_{\gamma_d(T)}$,  depending on the type of $T$.
As shown in Section \ref{subsec:number-of-elements},
$n_{\gamma_d(T)}$ is the minimum $n$ such that 
the number of elements of $L(n,d)$ of the type $\gamma_d(T)$ is positive.
\end{rem}

We now count the number of elements of $L(n,2)$.  
This is also needed to evaluate $\mu_{n,2}(\set{\emptyset})$. 
There are $\binom{n}{2}$ lines.
There are $n$ points of the first type and 
\[
\frac{1}{2} \binom{n}{2}\binom{n-2}{2}=3 \binom{n}{4}
\]
points of the second type. 
As discussed in Remark \ref{rem:n-existence}, 
this $3 \binom{n}{4}$ is positive
if and only if $n \ge 4$.

Therefore for $n\ge 3$ 
\begin{align}
\chi(\AAA_{n,2},t)
&=t^2 - \binom{n}{2} t + 3\binom{n}{4} + n(n-2)  \nonumber \\
&=t^2 - \binom{n}{2} t + 3\binom{n}{4} + 2 \binom{n}{2} - n, \nonumber \\
\mu_{n,2}(\set{\emptyset})
&= - 3\binom{n}{4}  -\binom{n}{2} + n-1.
\label{eq:d2-characteristic} 
\end{align}
These quantities are polynomials in $n$ 
and we prefer to write these polynomials
as integer combinations of binomial coefficients $\binom{n}{k}$.
Note that, 
in view of Remark \ref{rem:n-existence}, 
$\binom{n}{k}=0$ for integer $k>n$.

We now discuss the case of $d=3$.

We first look at rank two elements (lines) of $\AAA_{n,3}$.  
There are two types of elements. 
The first type  is an element of type $(2)\partitionof 2$. 
Each element of type $(2)$, such as $T=\set{\set{1,2}}$, 
corresponds to
the line connecting two points  as 
in the leftmost picture of Figure \ref{fig:d3r2}. 
$\set{\set{1,2}}$ is understood 
as the intersection of all hyperplanes $\set{\set{1,2,i}}$, $i=3,\dots,n$.
The second type is an element of type $(1,1)\partitionof 2$.
Each elements of type $(1,1)$ corresponds to
an intersection of two hyperplanes, 
such as $H(\set{\set{1,2,3}}) \cap H(\set{\set{4,5,6}})$.
As shown in the rightmost picture in Figure \ref{fig:d3r2},
two points ($p_3$ and $p_4$ in the picture) may overlap 
in this case without violating \ref{item:DefOfLattice1}
of Definition \ref{def:DefOfLattice}.
This type of element exists for $n \ge 5$ (cf.\ Remark \ref{rem:n-existence}).

\begin{figure}
\begin{center}
\includegraphics[width=4.5cm]{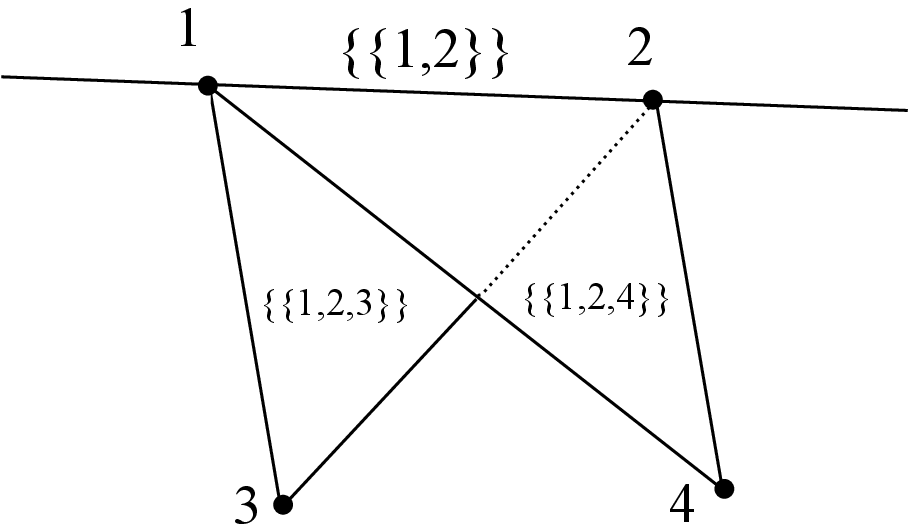}\qquad
\includegraphics[width=7cm]{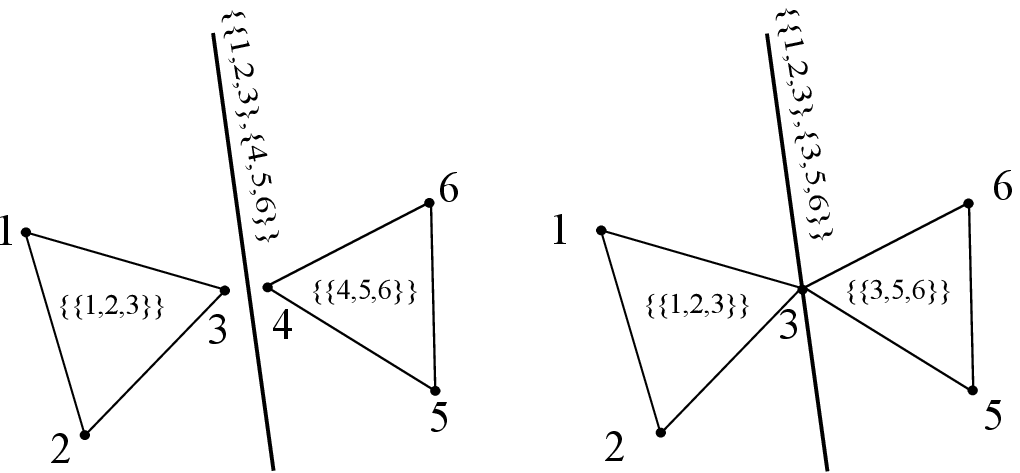}
\end{center}
\caption{Rank two elements for dimension three}
\label{fig:d3r2}
\end{figure}

Finally we look at rank three elements (points) of $\AAA_{n,3}$.
We will not repeat remarks on existence  of these elements of $L(n,3)$.
There are three types of rank three elements, 
corresponding to three partitions of $3$.
Each element $\set{\set{i}}$ of the first type $(3)\partitionof 3$, 
corresponds to
an original point in  $\PPP$.
Each element the second type $(2,1)\partitionof 3$  
corresponds to 
an intersection of a line of type $(2)\partitionof 2$ and a hyperplane, 
e.g.\ $H(\set{\set{1,2}}) \cap H(\set{\set{3,4,5}})$ 
as shown in Figure \ref{fig:d3r3-21}.
\begin{figure}
\begin{center}
\includegraphics[width=5cm]{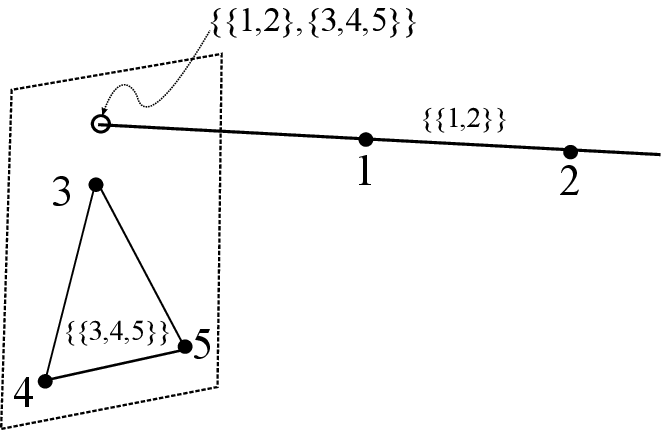}
\caption{Rank three  element for dimension three of type
 $(2,1)\partitionof 3$}
\label{fig:d3r3-21}
\end{center}
\end{figure}
\begin{figure}
\begin{center}
\includegraphics[width=13cm]{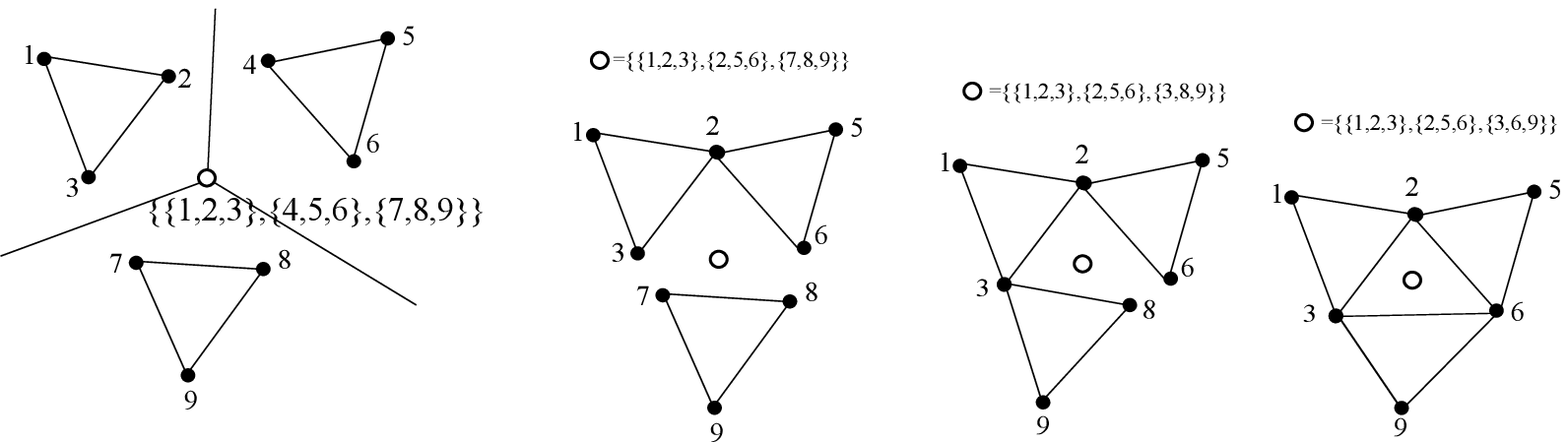}
\end{center}
\caption{Rank three elements for dimension three of type $(1,1,1)\partitionof 3$}
\label{fig:d3r3-111}
\end{figure}
The third type is $(1,1,1)\partitionof 3$,
corresponding to
an intersection of three hyperplanes
as depicted by a white circle in Figure \ref{fig:d3r3-111}.  
Without violating \ref{item:DefOfLattice1} 
of Definition \ref{def:DefOfLattice},
there are four patterns of overlaps of points.

As will be proved in Section \ref{sec:computation},
the M\"obius function depends only on the above types 
(i.e.\ the overlaps  of points do not affect the M\"obius function) 
and it is given as follows.
\begin{align}
\label{eq:mobius-d3}
\mu_{n,3}(\set{\set{1,2}})&=
-\mu_{n,3}(\set{\set{1,2},\set{3,4,5}})=\mu_{n-2,1}(\set\emptyset)
\\
&=n-3,\nonumber \\
\mu_{n,3}(\set{\set{1}})&=\mu_{n-1,2}(\set{\emptyset})\nonumber \\
&=   - 3 \binom{n-1}{4} - \binom{n-1}{2}+ n - 2,\nonumber
\end{align}
and $\mu_{n,3}(T)=(-1)^{\rho_3(T)}$ 
for all other $T$, $T\neq \set{\emptyset}$.

We need the numbers of elements of $L(n,3)$ 
to evaluate $\mu_{n,3}(\set{\emptyset})$.
These are tabulated in Table \ref{tab:d30}.
\begin{table}
\caption{Number of elements for $d=3$}
\label{tab:d30}
{\small
\begin{tabular}{|c|c|c|c|c|c|} \hline
(1) & (2) & (1,1) & (3) & (2,1) &  (1,1,1) \\ \hline
$\binom{n}{3}$ & $\binom{n}{2}$ & $10\binom{n}{6}+15\binom{n}{5}$ & $n$
& $10 \binom{n}{5}$ & $280 \binom{n}{9} + 840\binom{n}{8} + 630\binom{n}{7} + 120\binom{n}{6}$
\\ \hline
\end{tabular}
}
\end{table}
An element of a particular type exists 
if and only if the number of elements is positive in Table \ref{tab:d30}.  
For example, 
$T$ of type $(1,1)\partitionof 2$ exists 
if and only if $10\binom{n}{6}  + 15\binom{n}{5} > 0$, 
i.e.\ $n\ge 5$. 
From Table \ref{tab:d30} and \eqref{eq:mobius-d3}
we obtain (for $n\ge 4$) 
\begin{align*}
\chi(\AAA_{n,3},t)&= t^3 - \binom{n}{3} t^2 + 
\big[ -\binom{n}{2} + 3 \binom{n}{3} + 15 \binom{n}{5} + 10\binom{n}{6}\big] t\\
& \ -\big[ n - 2 \binom{n}{2} + 3 \binom{n}{3} + 35 \binom{n}{5} + 180\binom{n}{6}
+ 630 \binom{n}{7} \\
& \qquad\qquad\qquad + 840\binom{n}{8} + 280\binom{n}{9}\big], \\
\mu_{n,3}(\set\emptyset)&= -1 +n - \binom{n}{2}  + \binom{n}{3} + 20\binom{n}{5}
 + 170 \binom{n}{6} + 630\binom{n}{7}\\
& \qquad\qquad\qquad  + 840\binom{n}{8} + 280\binom{n}{9}.
\end{align*}

\section{Main result}
\label{sec:main}

In this section we show the following main theorem.
\begin{theorem}
\label{thm:structure of ideal is directprod}
Let $T\in L(n,d)$, $T\neq \emptyset$. 
Then the ideal $\ideal_{n,d}(T)$
is isomorphic to the direct product 
$\prod_{T_i \in T} \ideal_{n,d}(\{T_i\})$
as posets.  They are also isomorphic to 
$\prod_{T_i \in T} L(n-\numof{T_i},d-\numof{T_i})$.
\end{theorem}

The second part of this theorem is a  consequence of the following lemma.
\begin{lemma}
\label{lem:single-set-ideal}
For $\set{T_1}\in L(n,d)$, 
$\ideal_{n,d}(\{T_1\})$ is isomorphic to $L(n-\numof{T_1},d-\numof{T_1})$
as posets.
\end{lemma}

\begin{proof}
Suppose that $S=\set{S_1,\dots, S_l} \in 
\ideal_{n,d}(\set{T_1})$. Then 
$\set{S_1,\dots, S_l} \le \set{T_1}$,  so 
$S_i \supset T_1$, $\forall i$,  by
\eqref{eq:partial-order}.  Hence 
$\set{S_1\setminus T_1,\dots, S_l\setminus T_1 }\in L(n-\numof{T_1},d-\numof{T_1})$
by Remark \ref{remark:nonempty-intersecton}.
Therefore we have a map
\begin{align*}
\ideal_{n,d}(\set{T_1})\ni \set{S_1, \dots, S_l} \mapsto
\set{S_1\setminus T_1,\dots, S_l\setminus T_1 }\in L(n-\numof{T_1},d-\numof{T_1}),
\end{align*}
which is seen to be one-to-one and onto, and preserves the partial order.
\end{proof}

To prove the first part of Theorem \ref{thm:structure of ideal is directprod},
we show one proposition and three lemmas.

\begin{proposition}
\label{lemma:uniqueex}
Let $T< T' \in  L(n,d)$.
For each $T_i \in T$, 
there uniquely exists $T'_j \in T'$ such that $T'_j \subset T_i$.
\end{proposition}
\begin{proof}
It suffices to show the uniqueness.
Let $T_i \in T$ and $T'_j, T'_k \in T'$, $j\neq k$, 
satisfy $T'_j, T'_k \subset T_i$.
This means $T'_j\cup T'_k \subset T_i$.
Since $D_d(\Set{T'_j,T'_k})>0$, 
by  \eqref{eq:D-for-2},
$\numof{T_i}\geq\numof{(T'_j \cup T'_k)} > d+1$.
This conflicts with $\numof{T_i}\leq d$.
\end{proof}

\begin{lemma}
\label{lemma:3.2}
Let $T=\Set{T_1,\ldots,T_l}$ and $S=\Set{S_1,\ldots,S_l}$.
If $T_i \subset S_i$ for all $i$, then $D_d(S)\geq D_d(T)$.
\end{lemma}
\begin{proof}
Let $S'_i =S_i\setminus T_i$. 
Then by \eqref{eq:Dd-alternative}
\begin{align*}
D_d(S)-D_d(T)
&=\sum_{i=1}^l(\numof{S_i}- \numof{T_i}) 
+\numof{\bigcap_{i=1}^l T_i}-\numof{\bigcap_{i=1}^l S_i}\\
&=\sum_{i=1}^l\numof{S'_i}
+\numof{\bigcap_{i=1}^l T_i}-\numof{\bigcap_{i=1}^l S_i}.
\end{align*}
Since 
$\bigcap_i S_i=\bigcap_{i} (S'_i\cup T_i ) 
=(\bigcap_i T_i) 
\cup \big(S'_1\cap\bigcap_j S_j\big)
\cup \big(S'_2\cap\bigcap_j S_j\big)
\cup\cdots
\cup \big(S'_l\cap\bigcap_j S_j\big)$,
\begin{align*}
\numof{\bigcap_{i=1}^l S_i}
\leq 
\numof{\bigcap_{i=1}^l T_i}
+\numof{(S'_1\cap\bigcap_{j=1}^l S_j)}
+\cdots+\numof{(S'_l\cap\bigcap_{j=1}^l S_j)}.
\end{align*}
This implies
\begin{align*}
D_d(S)-D_d(T)
&\geq \sum_{i=1}^l\numof{S'_i}+\numof{\bigcap_{i=1}^l T_i}
-\numof{\bigcap_{i=1}^{l}T_i}
-\sum_{i=1}^l \numof{(S'_{i}\cap\bigcap_{j=1}^l S_j)}\\
&= \sum_{i=1}^l\numof{S'_i}
-\sum_{i=1}^l \numof{(S'_i\cap\bigcap_{j=1}^l S_j)}\\
&= \sum_{i=1}^l\numof{(S'_i\setminus \bigcap_{j=1}^l S_j)}.
\end{align*}
Hence $D_d(S)-D_d(T)\geq 0$.
\end{proof}

\begin{lemma}
\label{lemma:composition}
Let $T=\Set{T_1,\ldots,T_l}$,
$S^{(1)}=\Set{S^{(1)}_{1},\ldots,S^{(1)}_{m_1}}$,
$S^{(2)}=\Set{S^{(2)}_{1},\ldots,S^{(2)}_{m_2}}$, \ldots, 
$S^{(l)}=\Set{S^{(l)}_{1},\ldots,S^{(l)}_{m_l}}$, and
$S=S^{(1)}\cup\cdots\cup S^{(l)}$.
Assume $T_i\subset S^{(i)}_{j}$ for all $i,j$.
If $D_d(T)>0$ and $D_d(S^{(i)})>0$ for all $i$,
then $D_d(S)>0$.
\end{lemma}
\begin{proof}
Let $m=\sum_{i=1}^{l} m_i$. 
Then
\[
-(m-1)(d+1)=-(l-1)(d+1)-\sum_{i=1}^{l}(m_i-1)(d+1).
\] 
Hence 
\begin{align*}
&D_d(S)\\
&=-(m-1)(d+1) + \sum_{i=1}^l\sum_{j=1}^{m_i}\numof{S^{(i)}_{j}}
-\numof{\bigcap_{i=1}^{l}\bigcap_{j=1}^{m_i}S^{(i)}_{j}},\\
&=-(l-1)(d+1) -\sum_{i=1}^l (m_i -1)(d+1)
+\sum_{i=1}^l\sum_{j=1}^{m_i}\numof{S^{(i)}_{j}}
-\numof{\bigcap_{i=1}^{l}\bigcap_{j=1}^{m_i}S^{(i)}_{j}}\\
&=\sum_{i=1}^l (-(m_i -1)(d+1)+\sum_{j=1}^{m_i}\numof{S^{(i)}_{j}})
-(l-1)(d+1)-\numof{\bigcap_{i=1}^{l}\bigcap_{j=1}^{m_i}S^{(i)}_{j}}.
\end{align*}
Since 
$D_d(S^{(i)})+ \numof{\bigcap_{j=1}^{m_i}S^{(i)}_{j}}=-(m_i -1)(d+1)+\sum_{j=1}^{m_i}\numof{S^{(i)}_{j}}$,
\begin{align*}
D_d(S)
&=\sum_{i=1}^l D_d(S^{(i)})-(l-1)(d+1)
+\sum_{i=1}^l \numof{\bigcap_{j=1}^{m_i}S^{(i)}_{j}}
-\numof{\bigcap_{i=1}^{l}\bigcap_{j=1}^{m_i}S^{(i)}_{j}}\\
&=\sum_{i=1}^l D_d(S^{(i)})+
D_d\Big(\Set{\bigcap_{j=1}^{m_1}S^{(1)}_{j},\ldots, \bigcap_{j=1}^{m_l}S^{(l)}_{j}}\Big).
\end{align*}
Since $T_i \subset \bigcap_{j=1}^{m_i} S^{(i)}_{j}$,
it follows from Lemma \ref{lemma:3.2} that
\begin{align*}
D_d(S)&\geq\sum_{i=1}^l D_d(S^{(i)})+D_d(T)> 0.
\end{align*}
\end{proof}

\begin{lemma}
\label{thm:structure of ideal}
Let $T\in L(n,d)$, $T_1 \in T$ and $T' = T\setminus \Set{T_1}$.
Then $\ideal_{n,d}(\Set{T_1})\times \ideal_{n,d}(T')$
and $\ideal_{n,d} (T)$ are isomorphic as posets.
\end{lemma}

\begin{proof}
Let $T\in L(n,d)$, $T_1 \in T$ and $T' = T\setminus \Set{T_1}$.
For $(S,S')\in \ideal_{n,d}(\Set{T_1})\times \ideal_{n,d}(T')$,
let us define $\varphi(S,S')=S\cup S'$.
Then $\varphi(S,S') \in L(n,d)$ by Lemma \ref{lemma:composition}.
By definition $\varphi(S,S') \leq T$. 
Hence $\varphi$ is a map 
from $\ideal_{n,d}(\Set{T_1})\times \ideal_{n,d}(T')$
to $\ideal_{n,d} (T)$.
Moreover, 
if $(S,S')$ and $(S'',S''')$ satisfy  $S\leq S''$ and $S'\leq S'''$, 
then $\varphi(S,S')\leq \varphi(S'',S''')$.
On the other hand,
we can define the following map $\psi$ 
from $\ideal_{n,d} (T)$
to $\ideal_{n,d}(\Set{T_1})\times \ideal_{n,d}(T')$:
\begin{align*}
\psi(S)=(\Set{S_i|T_1 \subset S_i},\Set{S_i | T_1 \not\subset S_i}),
\end{align*}
which is the inverse map of $\varphi$.
Hence $\ideal_{n,d}(\Set{T_1})\times \ideal_{n,d}(T')$ and $\ideal_{n,d} (T)$
are isomorphic as posets.
\end{proof}

Applying Lemma \ref{thm:structure of ideal} recursively,
we have Theorem \ref{thm:structure of ideal is directprod}.

\section{Computation of M\"obius function and the characteristic polynomial}
\label{sec:computation}

In this section 
we apply Theorem \ref{thm:structure of ideal is directprod} 
to compute the M\"obius function and the characteristic polynomial 
of the intersection lattice $L(n,d)$ for $d\le 6$.  
This section is divided into four subsections.  

In Section \ref{subsec:mobius} 
we derive an explicit formula 
for the value of the M\"obius function of $L(n,d)$ 
and show that it only depends on the type of $T\in L(n,d)$.  
Next in Section \ref{subsec:number-of-elements} 
we derive a formula for the number of elements of the same type as $T\in L(n,d)$.
Then in Section \ref{subsec:identities} 
we derive some identities for these numbers, 
which are useful for checking the results of computations by computer.  
Finally in Section \ref{subsec:tables} 
we present lists of the numbers of elements and the characteristic polynomials
for $d\le 6$.

\subsection{M\"obius function of the intersection lattice}
\label{subsec:mobius}

We first obtain the value of M\"obius function of $L(n,d)$.

\begin{proposition}
\label{thm:decomposition}
For $T\in L(n,d)$, $T\neq \emptyset$,
\begin{align*}
\mu_{n,d}(T)=
\prod_{T_i \in T}\mu_{n-\numof{T_i},d-\numof{T_i}}(\Set{\emptyset}).
\end{align*}
\end{proposition}
Note that for $T=\emptyset$ we have  $\mu_{n,d}(\emptyset)=1$.
Also, 
as discussed at the beginning of Section \ref{subsec:illustration}, 
$\mu_{n-\numof{T_i},d-\numof{T_i}}(\Set{\emptyset})=-1$
if $d=\numof{T_i}$.

Proposition \ref{thm:decomposition} 
is an immediate
consequence of Theorem \ref{thm:structure of ideal is directprod} and
the following well-known lemma.

\begin{lemma}[Proposition 3.8.2 of \cite{stanley-book-1}]
Let $P$ and $P'$ be posets, and  $P\times P'$ the direct product of
 posets $P$ and $P'$.
Then $\mu_P(S,T)\cdot\mu_{P'}(S',T')=\mu_{P\times P'}((S,S'),(T,T'))$
for $S,T\in P$ and $S',T' \in P'$,
where $\mu$ denotes the M\"obius function for each poset.
\end{lemma}

Proposition \ref{thm:decomposition} shows that 
the M\"obius function of $L(n,d)$ 
is completely determined by 
the values of $\mu_{n+k-d,k}(\Set{\emptyset})$, $0 \le k\le d$.
In particular for $T\neq \Set{\emptyset}$, 
$\mu_{n,d}(T)$ is 
a product of  $\mu_{n+k-d,k}(\Set{\emptyset})$ for $k$ smaller than $d$.
As seen in the examples of Section \ref{subsec:illustration},
$\mu_{n',d'}(\set{\emptyset})$ is a polynomial in $n'$.
Hence $\mu_{n,d}(T)$, $T\neq\set{\emptyset}$, 
can be immediately obtained from $\mu_{n',d'}(\set{\emptyset})$ for $d' < d$. 
Therefore for the recursion on $d$, 
the essential step is 
to compute  $\mu_{n,d}(\set{\emptyset})$ by \eqref{eq:mumaxphi}, 
which will be discussed in the  next subsection.

As a corollary to Proposition \ref{thm:decomposition}
we have the following result.  
\begin{cor}
\label{cor:type}
Let
$T=\Set{T_1,\ldots,T_l}\in L(n,d)$ and
$T'=\Set{T'_1,\ldots,T'_l}\allowbreak \in L(n,d')$
satisfy $\codim_d(T_i)=\codim_{d'}(T'_i)$ for each $i$.
Define $\bar \mu_{u,d}(T)=\mu_{d+u,d}(T)$, $u\ge 1$.
Then
\begin{align*}
\bar\mu_{u,d}(T)=\bar\mu_{u,d'}(T').
\end{align*}
\end{cor}

In this sense 
the value of the M\"obius function 
depends only on the multiset of codimensions,
i.e., the type  $\gamma_d(T)$ of $T$.
Therefore from now on 
we denote $\mu_{n,d}(T)=\mu_{n,d}(\gamma)$ if $\gamma_d(T)=\gamma$.

\subsection{Number of elements of the intersection lattice}
\label{subsec:number-of-elements}

The results of the previous subsection implies that 
the terms of the summations in  \eqref{eq:characteristic-function} 
and \eqref{eq:mumaxphi} 
can be grouped into different types.  
Then the question is 
how to obtain  the number of elements of the same type in $L(n,d)$, 
denoted by $\lambda_{n,d}(\gamma)$ below.   
In this subsection we give an explicit expression
for $\lambda_{n,d}(\gamma)$  in Proposition \ref{prop:number-type}.

Let 
\[
\lambda_{n,d}(\gamma) = \numof{\set{T \in L(n,d)| \gamma_d(T)=\gamma}}
\]
denote the number of $T\in L(n,d)$ of type $\gamma$.
Then \eqref{eq:characteristic-function} and \eqref{eq:mumaxphi} 
are written as follows.
\begin{align}
\chi_{n,d}(t) 
&=
 \mu_{n,d}(\set{\emptyset})+
\sum_{i=0}^{d}
\sum_{\gamma \partitionof i} 
\lambda_{n,d}(\gamma)\mu_{n,d}(\gamma) t^{d+1-i}
, 
\label{eq:characteristic-type}\\
  \mu_{n,d}(\set{\emptyset})
&=- 
\sum_{i=0}^{d}
\sum_{\gamma \partitionof i} 
 \lambda_{n,d}(\gamma)\mu_{n,d}(\gamma).
\label{eq:mumaxphi-type}
\end{align}

For stating  Proposition \ref{prop:number-type}  we need some more definitions.
For 
a partition 
$\gamma=(\gamma_1, \dots, \gamma_l)$
of a nonnegative integer
let 
\[
\Stab_{S_l}(\gamma)
=\Set{\sigma\in S_l|\gamma_i=\gamma_{\sigma(i)}, i=1,\dots,l}
\]
denote the stabilizer of the symmetric group $S_l$ fixing $\gamma$.
Then we have
\begin{align*}
\numof{\Stab_{S_l}(\gamma)}
=\prod_{k} m_k(\gamma) !,
\end{align*}
where 
$m_k(\gamma)$ denotes 
the multiplicity $\numof{\Set{i|\gamma_i=k}}$ 
of $k\in \set{1,\dots,d}$ in $\gamma$.

Denote the elements of $2^{\Set{1,\ldots,l}}$ as
\[
2^{\Set{1,\ldots,l}}
=\Set{\emptyset,\set{1},\dots, \set{l},\set{1,2},\dots,\set{1,\dots,l}}
=\set{I_1, \dots, I_{2^l}},
\]
where $I_1=\emptyset$.  
Let  $N(n,d;\gamma)$, $\gamma=(\gamma_1, \dots, \gamma_l)$, $l\ge 1$, 
denote the set of maps $\nu$ 
from the power set $2^{\Set{1,\ldots,l}}$ 
to the set $\NN$ of nonnegative integers satisfying the following:
\begin{enumerate}[label=\arabic{*}) ]
\item \label{item:def-N-A}
$\sum_{I\colon i\in I} \nu(I) = d+1-\gamma_i$ for all $i=1,\dots,l$.
\item $\sum_{I'\colon I\subset I'} \nu(I')  <  d+1 - \sum_{i\in I}\gamma_i$ for all $I$ such that $\numof{I}\ge 2$.
\item \label{item:def-N-C}
$\sum_{I\in 2^{\Set{1,\ldots,l}}}\nu(I)=n$.
\end{enumerate}
\begin{example}
In the case  when $\gamma=(d+1)$, 
$\nu$ is a map from $2^{\set{1}}=\Set{ \emptyset , \set{1}}$ to $\NN$.
By the condition \ref{item:def-N-A},
$\sum_{I\colon 1\in I} \nu(I)=\nu(\set{1})=0$.
Hence, by \ref{item:def-N-C},
$\nu(\emptyset)=n$.
$N(n,d;(d+1))$ consists of this $\nu$ only.
\end{example}
\begin{remark}
Consider the elements of $\set{1,\dots,n}$ as ``symbols''. 
Each $T_i$ in $T = \set{T_1,\dots,T_l}$ is a subset of $\set{1,\dots,n}$ 
and therefore contains $\numof{T_i}$ symbols.
Also call $T_i$ a ``block''.  
We can think of $T = \set{T_1,\dots,T_l}$ 
as putting symbols $1,\dots,n$ in the blocks $T_1,\dots, T_l$.  
Some symbol appears in several blocks.  
For $I\subset \set{1,\dots,l}$, 
$\nu(I)$ denotes the number of symbols commonly contained in $T_i$, $i\in I$, 
but not contained in any other $T_i$, $i\not\in I$.  
The condition 3) on $\nu$ means that
$n$ symbols $1,\dots,n$ are  classified by the blocks containing them.
The condition 1) on $\nu$ corresponds 
to the size of each block $\numof{T_i}=d+1-\gamma_i$.
The condition 2) on $\nu$ is essential and 
corresponds to 1) of Definition \ref{def:DefOfLattice}.
\end{remark}

\begin{example}
First let us consider the case when $\gamma=(\gamma_1)$.  In this case
\begin{equation}
\label{eq:one-part}
\lambda_{n,d}((\gamma_1))
=\binom{n}{d+1-\gamma_1}.
\end{equation}

Next let us consider the case when  $\gamma=(\gamma_1,\gamma_2)$.
Let $t_i=d +1 -\gamma_i$, $i=1,2$.
For an element $T=\Set{T_1,T_2}$ of type $\gamma$, by
\eqref{eq:D-for-2},
\begin{align*}
0 &\leq \numof{(T_1 \cap T_2)} 
= \numof{T_1}+\numof{T_2} - \numof{(T_1 \cup T_2)}\\
& \leq  \numof{T_1}+\numof{T_2}
-d -2
=t_1+t_2-d-2=d-\gamma_1-\gamma_2.
\end{align*}
It also follows by definition that $\numof{(T_1 \cap T_2)} \leq \min(t_1,t_2)$.
Write $\nu=\numof{(T_1 \cap T_2)}$.
If $\gamma_1>\gamma_2$, then 
$(\codim(T_2),\codim(T_1))\neq (t_1,t_2)$
as ordered pairs.
Hence, for the case $\gamma_1 > \gamma_2$, 
\begin{align}
\label{eq:two-part-1}
\lambda_{n,d}((\gamma_1, \gamma_2)) 
=
\sum_{\nu=0}^{\min(t_1, t_2, d-\gamma_1 -\gamma_2)}
\binom{n}{\nu}\binom{n-\nu}{t_1-\nu}\binom{n-t_1}{t_2-\nu}.
\end{align}
On the other hand,
if $\gamma_1=\gamma_2$, then 
$(\codim(T_2),\codim(T_1))= (\gamma_1,\gamma_2)$
as ordered pairs
for all elements $T=\Set{T_1,T_2}$ of type $\gamma$.
Since $T_1\neq T_2$, for the case  $\gamma_1 = \gamma_2$, 
\begin{align}
\label{eq:two-part-2}
\lambda_{n,d}((\gamma_1, \gamma_2)) 
=
\frac{1}{2}\sum_{\nu=0}^{\min(t_1, t_2, d-\gamma_1 -\gamma_2)}
\binom{n}{\nu}\binom{n-\nu}{t_1-\nu}\binom{n-t_1}{t_2-\nu}.
\end{align}
\end{example}

Now we present the following proposition.
\begin{proposition}
\label{prop:number-type}
For $\gamma\neq (0)$,
$\lambda_{n,d}(\gamma)$  
is given  as follows.
\begin{align}
\label{eq:number-type}
\lambda_{n,d}(\gamma)
&=\frac{1}{\prod_{k=1}^d m_k(\gamma) !}
\sum_{\nu\in N(n,d;\gamma)} 
\frac{n!}{\nu(I_1)! \nu(I_2)!\cdots \nu(I_{2^l})!}
\\
&=\frac{1}{\prod_{k=1}^d m_k(\gamma) !}
\sum_{\nu\in N(n,d;\gamma)} 
\frac{(\nu(I_2)+ \dots + \nu(I_{2^l}))!}{\nu(I_2)! \cdots \nu(I_{2^l})!}
\binom{n}{\nu(I_2)+ \dots + \nu(I_{2^l})}. \nonumber
\end{align}
\end{proposition}

Before giving a proof of this proposition, 
we give some explanation on 
the range of summation in \eqref{eq:number-type}.
We can consider $(\nu(I_1),\dots,\nu(I_{2^l}))$ as 
a $2^l$-dimensional vector of non-negative integers.  
The equalities and the inequalities
in 1),2),3) for $\nu$ specify a polytope.  
Hence $N(n,d;\gamma)$ can be identified 
with the  set of integer points  in a polytope in $\RR^{2^l}$.
Since the dimension $2^l$ of the vector increases exponentially with $l$, 
the number of terms in \eqref{eq:number-type} 
increases doubly exponentially in $l$.
In our computation for $d=6$ and $\gamma=(1,1,1,1,1,1)$,
$\numof N(n,6;(1,1,1,1,1,1))=109719496370$.  Computing a sum of 
this many polynomials is quite heavy.

The equalities and inequalities in 1),2) for $\nu$ 
concern only $\nu(I)$, $I\neq \emptyset$, 
and the bounds for these nonnegative integers 
are given in terms of $\gamma$ and $d$ only.
Therefore the range for $\nu(I), I\neq \emptyset$, 
in  $N(n,d;\gamma)$ does not depend on $n$. $n$ only appears through 
3): 
\[
\nu(\emptyset)=\nu(I_1)= n - (\nu(I_2)+\dots+\nu(I_{2^l})).
\]
Therefore in the right-hand side of  \eqref{eq:number-type} 
the sum is a finite sum not depending on $n$ and $n$ only appears in
the binomial coefficient $\binom{n}{\nu(I_2)+ \dots + \nu(I_{2^l})}$.

Now we give a proof of  Proposition \ref{prop:number-type}.
\begin{proof}[Proof of Proposition \ref{prop:number-type}]
Let us consider 
$l$-tuples $(T_1,\ldots,T_l)$ of subsets of $\Set{1,\ldots,n}$.
We define $\tilde L(n,d;\gamma)$ 
to be the set of  $l$-tuples $(T_1,\ldots,T_l)$ 
of subsets in $\Set{1,\ldots,n}$
satisfying the following:
\begin{enumerate}[label=\arabic{*}) ]
\item $\Set{T_1,\ldots,T_l} \in L(n,d)$.
\item $\codim_d{T_i} = \gamma_i$ for each $i$.
\end{enumerate}
Let $(T_1,\ldots,T_l)\in\tilde L(n,d;\gamma)$.
Then, 
since  $\Set{T_1,\ldots,T_l} \in L(n,d)$,
by \eqref{eq:D-for-2},
$\numof{(T_i\cup T_j)}> d $ for $i\neq j$. 
Since $\numof{T_i}$ and  $\numof{T_j}$ are less than or equal to $d+1$,
we obtain  $T_i\neq T_j$.
Therefore, 
for 
$(T_1,\ldots,T_l)\in\tilde L(n,d;\gamma)$ and 
$\sigma\in \Stab_{S_l}(\gamma)$,
we have 
\begin{align*}
(T_1,\ldots,T_l)\neq (T_{\sigma(1)},\ldots,T_{\sigma(l)}) 
\in\tilde L(n,d;\gamma)
\end{align*}
if $\sigma$ is not the identity.
This implies 
\begin{equation}
\label{eq:tildeL}
\numof{\tilde L(n,d;\gamma)}
=\lambda_{n,d}(\gamma)\cdot\numof{\Stab_{S_l}(\gamma)}
= \lambda_{n,d}(\gamma)\cdot\prod_{k} m_k(\gamma) ! .
\end{equation}

For 
$T=(T_1,\ldots,T_l)\in\tilde L(n,d;\gamma)$ 
and a subset $I \subset \Set{1,\ldots,l}$,
define $\tau(T,I)$ by
\begin{align*}
\tau(T,I)&=\Set{t| t\in T_i \iff i \in I}\\
&=\bigcap_{i\in I} T_i \setminus \bigcup_{i\not\in I} T_i .
\end{align*}
Moreover, for $\nu \in N(n,d;\gamma)$
let us define $\tilde L(n,d;\gamma,\nu)$ by
\begin{align*}
\tilde L(n,d;\gamma,\nu)=
\Set{ T\in  \tilde L(n,d;\gamma) | 
\forall I \subset \Set{1,\ldots,l}, 
\nu(I) = \numof{\tau(T,I)}
}.
\end{align*}
Then, by definition,
we have the following decomposition of $\tilde L(n,d;\gamma)$:
\begin{align*}
\tilde L(n,d;\gamma)=\coprod_{\nu\in N(n,d;\gamma)}\tilde L(n,d;\gamma,\nu).
\end{align*}
Now note that 
\begin{align*}
\numof{\tilde L(n,d;\gamma,\nu)}
=\frac{n!}{\nu(I_1)! \nu(I_2)! \cdots \nu(I_{2^l})!}.
\end{align*}
Therefore 
\[
\numof{\tilde L(n,d;\gamma)}=
\sum_{\nu\in N(n,d;\gamma)}
\frac{n!}{\nu(I_1)! \nu(I_2)! \cdots \nu(I_{2^l})!}.
\]
This together with \eqref{eq:tildeL}  proves the proposition.
\end{proof}

\subsection{Identities for the number of elements}
\label{subsec:identities}

We have coded the finite sum in \eqref{eq:number-type} in a computer program 
and evaluated $\lambda_{n,d}(\gamma)$ up to $d=6$.  
In the next subsection
we present our computational results.
However the range of summation in \eqref{eq:number-type} is 
somewhat complicated and our code was error-prone.
Therefore it is desirable to have some way of checking our results.  
Here we present
some identities among $\lambda_{n,d}(\gamma)$'s, 
which can be used for checking purposes.

Again we need some more definitions for stating the identities.
Let $T_i$ be a subset of $\set{1,\dots,n}$ of size $d=\numof{T_i}$.  
Then $H(\set{T_i})$ is a hyperplane of $\AAA_{n,d}$.
By abuse of terminology, 
we also call $\set{T_i}$ itself a hyperplane.
Let $\set{T_1}, \ldots, \set{T_m}$ be $m$ distinct hyperplanes.   
Consider the intersection  $H(\set{T_1})\cap \dots \cap H(\set{T_m})$
of corresponding hyperplanes of $\AAA_{n,d}$,
or equivalently the join  of these hyperplanes
$\set{T_1}, \dots, \set{T_m}$ in $L(n,d)$:
\[
\set{T_1}\vee \dots \vee \set{T_m} \in L(n,d).
\]
It seems hard to explicitly describe $S_1, \dots, S_{l'}$ such that  
$S=\{S_1, \dots,S_{l'}\}=\set{T_1}\vee \cdots \vee \set{T_m}$.  
However we can count 
the number of 
$\set{T_1, \dots, T_m}$, ($\numof T_i = d$, $\forall i$),
such that 
$\set{T_1}\vee \dots \vee \set{T_m}$ 
is an element of a  particular type of $L(n,d)$.  
This will give us the desired identities.

For a particular $T\in L(n,d)$ define
\begin{align*}
&\nt_{n,d}(m,T)\\
&=
\numof{ \Set{ \set{T_1, \dots, T_m} |  
  \begin{array}{c}
    \set{T_1}, \dots, \set{T_m} : \text{distinct hyperplanes}, \\
    T= \set{T_1}\vee \dots \vee \set{T_m}  
  \end{array}
}}, 
\end{align*}
which is the number of ways of choosing $m$ distinct hyperplanes 
such that their join is $T$.
By Theorem \ref{thm:structure of ideal is directprod}, 
$\nt_{n,d}(m,T)$ only depends on the type $\gamma_d(T)$ of $T$.  
Hence we can write
\[
\nt_{n,d}(m,T)=\nt_{n,d}(m,\gamma) \ \ \text{if}\ \ \gamma_d(T)=\gamma.
\]
Note that 
there are $\binom{\binom{n}{d}}{m}$ ways to choose $m$ distinct hyperplanes 
from $\set{1,\dots,n}$.
Therefore we have the following identity:
\begin{equation}
\label{eq:identity-sum}
\binom{\binom{n}{d}}{m}=\nt_{n,d}(m,\set{\emptyset}) + 
\sum_{i=0}^{d}
\sum_{\gamma \partitionof i} 
\lambda_{n,d}(\gamma) \cdot \nt_{n,d}(m,\gamma).
\end{equation}
If we can compute $\nt_{n,d}(m,\gamma)$,
these identities for various $m$ 
can be used to check computations of $\lambda_{n,d}(\gamma)$.
Hence it remains to show how to evaluate $\nt_{n,d}(m,\gamma)$, 
which is again based on recursion on $d$.

First we consider $\nt_{n,d}(m,\gamma)$ for some special $\gamma$.
Write  $(1^h)=(\underbrace{1,1,\dots,1}_{h})$.  Then
\[
\nt_{n,d}(m,(1^h))=\delta_{mh},
\]
where $\delta_{mh}$ is Kronecker's delta.
Also note that 
\[
\nt_{n,d}(m,T)=0 \ \ \text{ if } \ \ 
\rho_d(T)>m.
\]
In particular 
\[
\nt_{n,d}(m,\set{\emptyset})=0 \ \ \text{for}\ \  1\le m\le d.
\]
Based on these observations, 
there are two uses of \eqref{eq:identity-sum}.  
For $m=1,\dots,d$,  
we can use \eqref{eq:identity-sum} 
to check $\lambda_{n,d}(\gamma)$ for 
$\gamma\partitionof i \leq d$.
With $m>d$, 
\eqref{eq:identity-sum} gives the values of $\nt_{n,d}(m,\set{\emptyset})$.

Now we show 
how $\nt_{n,d}(m,\gamma)$ is evaluated 
from  $\nt_{n,d'}(m,\set{\emptyset})$ with $d' < d$. 
We list $\nt_{n,d}(m,\set{\emptyset})$ for $d=0,1,2$.
For $d=0$ 
we define $\nt_{n,0}(m,\set{\emptyset})=1$.
For $d=1$, 
since the intersection of more than one point is empty, 
$\nt_{n,1}(m,\set{\emptyset})=\binom{n}{m}$ for $m> 1$.
For $d=2$, the intersection of $m>2$ lines is non-empty 
if and only if they contain a common point $p_i$.  
Therefore for $m>2$, 
\[
\nt_{n,2}(m,\set{\emptyset})=\binom{\binom{n}{2}}{m} - n \binom{n-1}{m}.
\]

Finally as  another consequence of the main theorem 
we have the following proposition.
It allows us to evaluate $\kappa_{n,d}(m,\gamma)$ recursively 
from $\kappa_{n,d'}(m,\set{\emptyset})$,  $d' < d$.  

\begin{proposition}
\label{prop:identity}
For $T=\set{T_1,\dots,T_l}\in L(n,d)$, $T \neq \set{\emptyset}$, 
and a positive integer $m$,
define 
\begin{align*}
M(m,T) &= 
\Set{ (m_1,\dots, m_l) \in \ZZ_{>0}^l |
\begin{array}{c}
m_i \ge \codim_d (T_i), \forall i.\\
m_i=1 \text{ for } \codim_d (T_i)=1.\\
m=m_1+\dots+m_l.
\end{array}
}.
\end{align*}
Then 
\begin{align*}
\nt_{n,d}(m,T) =\sum_{(m_1,\dots,m_l)\in {M(m,T)}}\ \prod_{i=1}^l
\nt_{n-\numof{T_1},d-\numof{T_1}}(m_i,\set{\emptyset}). 
\end{align*}
\end{proposition}

Note that in the product a term with $d=\numof{T_i}$  does not 
contribute to the product since $\nt_{n,0}(m,\set{\emptyset})\equiv 1$. 
We omit a  detailed proof of the proposition.

\subsection{Number of elements and the characteristic polynomial up to dimension six}
\label{subsec:tables}

In this section we present our computational results for $4\le d\le 6$,
since the cases $d\le 3$ were already discussed in Section \ref{subsec:illustration}.
We just recall
\begin{align*}
\mu_{n,0}(\set{\emptyset})& =-1,  \qquad  \mu_{n,1}(\set{\emptyset})=n-1,\\
\mu_{n,2}(\set{\emptyset})&=    - 3\binom{n}{4}  -\binom{n}{2} + n-1 .
\end{align*}
From now on, to save space, we use the following abbreviated notation.  
\[
n_k = \binom{n}{k}.
\]
Then, for example,  
$\mu_{n,3}(\set\emptyset)$ is displayed  as
\[
\mu_{n,3}(\set\emptyset)= -1+n - n_2  + n_3  + 20 n_5
 + 170 n_6 + 630 n_7 +  840 n_8 + 280n_9.
\]

We now present the computational results for $d=4$.  
Because of   \eqref{eq:one-part}, \eqref{eq:two-part-1}, \eqref{eq:two-part-2}, 
we only show $\lambda_{n,d}((\gamma_1, \dots, \gamma_l))$ where $l\ge 3$.
Also, for further notational simplification,
we omit the subscripts and write e.g. $\lambda(1,1)$ instead of $\lambda_{n,4}((1,1))$.
\small
\begin{align*}
\lambda(1,1,1)& = 
15 n_{6}
+1470 n_{7}
+11340 n_{8}
+30240 n_{9}
+37450 n_{10}
+23100 n_{11}
+5775 n_{12}
, \\
\lambda(2,1,1) & = 
1260 n_{7}
+10080 n_{8}
+23940 n_{9}
+21000 n_{10}
+5775 n_{11}
, \\
\lambda(1,1,1,1) &= 
2100 n_{7}
+120855 n_{8}
+1640520 n_{9}
+9585450 n_{10}
+29799000 n_{11}
\\ & \quad
+54365850 n_{12}
+60660600 n_{13}
+41166125 n_{14}
+15765750 n_{15}
+2627625 n_{16}
.
\end{align*}

\begin{align*}
\chi(\AAA_{n,4},t)&=
t^4
-n_{4} t^3
+
\Big[
-n_{3}
+4 n_{4}
+45 n_{6}
+70 n_{7}
+35 n_{8}
\Big]t^2
\\ & \quad
+
\Big[
-n_{2}
+3 n_{3}
-6 n_{4}
-180 n_{6}
-1995 n_{7}
-11620 n_{8}
-30240 n_{9}
\\ & \qquad\qquad
-37450 n_{10}
-23100 n_{11}
-5775 n_{12}
\Big]t
\\ &\quad 
+
\Big[
-n
+2 n_{2}
-3 n_{3}
+4 n_{4}
+250 n_{6}
+8995 n_{7}
+184835 n_{8}
+1873620 n_{9}
\\ & \qquad\qquad 
+9963100 n_{10}
+30070425 n_{11}
+54435150 n_{12}
+60660600 n_{13}
\\ & \qquad\qquad 
+41166125 n_{14}
+15765750 n_{15}
+2627625 n_{16}
\Big] .
\end{align*}

\begin{align*}
\mu_{n,4}(\set{\emptyset})&=
-1
+n
-n_{2}
+n_{3}
-n_{4}
-115 n_{6}
-7070 n_{7}
-173250 n_{8}
-1843380 n_{9}
\\ & \quad 
-9925650 n_{10}
-30047325 n_{11}
-54429375 n_{12}
-60660600 n_{13}
\\ & \quad 
-41166125 n_{14}
-15765750 n_{15}
-2627625 n_{16} 
.
\end{align*}
\bigskip

\normalsize
The results for $d=5$ are as follows.
\footnotesize
\allowdisplaybreaks
\begin{align*}
\lambda(1,1,1)& = 
105 n_{7}
+9240 n_{8}
+102060 n_{9}
+453600 n_{10}
+1089550 n_{11}
+1561560 n_{12}
\\ & \qquad
+1336335 n_{13}
+630630 n_{14}
+126126 n_{15}
,\\
\lambda(2,1,1)&= 
105 n_{7}
+15960 n_{8}
+170100 n_{9}
+642600 n_{10}
+1166550 n_{11}
+1136520 n_{12}
\\ & \qquad
+585585 n_{13}
+126126 n_{14}
,\\
\lambda(1,1,1,1)& = 
42000 n_{8}
+2796255 n_{9}
+52475850 n_{10}
+464829750 n_{11}
+2391764760 n_{12}
\\ & \qquad
+7945667730 n_{13}
+18019621620 n_{14}
+28608004425 n_{15}
\\ & \qquad
+31876244400 n_{16}
+24459299865 n_{17}
+12318095790 n_{18}
\\ & \qquad
+3666482820 n_{19}
+488864376 n_{20}
,\\
\lambda(3,1,1)& = 
3360 n_{8}
+37800 n_{9}
+138600 n_{10}
+219450 n_{11}
+152460 n_{12}
+36036 n_{13}
,
\\
\lambda(2,2,1)& = 
5040 n_{8}
+56700 n_{9}
+201600 n_{10}
+300300 n_{11}
+194040 n_{12}
+45045 n_{13} 
,
\\
\lambda(2,1,1,1)& = 
47040 n_{8}
+3859380 n_{9}
+77275800 n_{10}
+682882200 n_{11}
+3311930160 n_{12}
\\ & \qquad
+9818128320 n_{13}
+18834816000 n_{14}
+23991267300 n_{15}
\\ & \qquad
+20272652400 n_{16}
+10985154180 n_{17}
+3473510040 n_{18}
+488864376 n_{19}
,\\
\lambda(1,1,1,1,1) &= 
70560 n_{8}
+28259280 n_{9}
+1892400300 n_{10}
+49372299900 n_{11}
\\ & \qquad
+678800152800 n_{12}
+5726202381900 n_{13}
+32397151296510 n_{14}
\\ & \qquad
+129991147035750 n_{15}
+383340007050000 n_{16}
+849257881311840 n_{17}
\\ & \qquad
+1429769976354720 n_{18}
+1833899747359680 n_{19}
+1780941069507600 n_{20}
\\ & \qquad
+1287845979720300 n_{21}
+672060801181770 n_{22}
+239171396233770 n_{23}
\\ & \qquad
+51946728593760 n_{24}
+5194672859376 n_{25}
.\end{align*}

\allowdisplaybreaks
\begin{align*}
\chi(\AAA_{n,5},t)&=
t^5 -n_{5} t^4
+
\Big[
-n_{4}
+5 n_{5}
+105 n_{7}
+280 n_{8}
+315 n_{9}
+126 n_{10}
\Big]t^3
\\ & \quad
+
\Big[
-n_{3}
+4 n_{4}
-10 n_{5}
-630 n_{7}
-11760 n_{8}
-105084 n_{9}
-454860 n_{10}
\\ & \qquad\qquad
-1089550 n_{11}
-1561560 n_{12}
-1336335 n_{13}
-630630 n_{14}
-126126 n_{15}
\Big]t^2
\\ & \quad
+
\Big[
-n_{2}
+3 n_{3}
-6 n_{4}
+10 n_{5}
+1540 n_{7}
+112371 n_{8}
+3739176 n_{9}
+57660120 n_{10}
\\ & \qquad\qquad
+479155600 n_{11}
+2413802160 n_{12}
+7965127170 n_{13}
+18028954944 n_{14}
\\ & \qquad\qquad
+28609896315 n_{15}
+31876244400 n_{16}
+24459299865 n_{17}
+12318095790 n_{18}
\\ & \qquad\qquad
+3666482820 n_{19}
+488864376 n_{20}
\Big]t
\\ & \quad
+
\Big[
-n
+2 n_{2}
-3 n_{3}
+4 n_{4}
-5 n_{5}
-1729 n_{7}
-444808 n_{8}
-51417954 n_{9}
\\ & \qquad\qquad
-2407629420 n_{10}
-54882065700 n_{11}
-712167312780 n_{12}
-5852028673491 n_{13}
\\ & \qquad\qquad
-32709595374456 n_{14}
-130517797815405 n_{15}
-383948623858800 n_{16}
\\ & \qquad\qquad
-849734640219465 n_{17}
-1430012864760480 n_{18}
-1833972588151704 n_{19}
\\ & \qquad\qquad
-1780950846795120 n_{20}
-1287845979720300 n_{21}
-672060801181770 n_{22}
\\ & \qquad\qquad
-239171396233770 n_{23}
-51946728593760 n_{24}
-5194672859376 n_{25}
\Big]
.
\end{align*}
\begin{align*}
\mu_{n,5}(\set{\emptyset})&=
-1
+n
-n_{2}
+n_{3}
-n_{4}
+n_{5}
+714 n_{7}
+343917 n_{8}
+47783547 n_{9}
\\ & \qquad
+2350424034 n_{10}
+54403999650 n_{11}
+709755072180 n_{12}
\\ & \qquad
+5844064882656 n_{13}
+32691567050142 n_{14}
+130489188045216 n_{15}
\\ & \qquad
+383916747614400 n_{16}
+849710180919600 n_{17}
+1430000546664690 n_{18}
\\ & \qquad
+1833968921668884 n_{19}
+1780950357930744 n_{20}
+1287845979720300 n_{21}
\\ & \qquad
+672060801181770 n_{22}
+239171396233770 n_{23}
+51946728593760 n_{24}
\\ & \qquad
+5194672859376 n_{25}
.
\end{align*}

\normalsize
\bigskip
Finally the results for $d=6$ are as follows.

\scriptsize
\begin{align*}
\lambda(1,1,1)& = 
420 n_{8}
+40600 n_{9}
+620550 n_{10}
+4158000 n_{11}
+16046800 n_{12}
+39399360 n_{13}
\\ & \qquad
+63588525 n_{14}
+67267200 n_{15}
+44900856 n_{16}
+17153136 n_{17}
+2858856 n_{18}
,\\
\lambda(2,1,1) &= 
840 n_{8}
+105210 n_{9}
+1486800 n_{10}
+8339100 n_{11}
+25225200 n_{12}
+46576530 n_{13}
\\ & \qquad
+54444390 n_{14}
+39414375 n_{15}
+16144128 n_{16}
+2858856 n_{17}
,\\
\lambda(1,1,1,1) &= 
105 n_{8}
+388080 n_{9}
+32389875 n_{10}
+847573650 n_{11}
+11095663425 n_{12}
\\ & \qquad
+88232164020 n_{13}
+470574214110 n_{14}
+1778211935500 n_{15}
+4911176169900 n_{16}
\\ & \qquad
+10078325056800 n_{17}
+15457185789045 n_{18}
+17651874149910 n_{19}
\\ & \qquad
+14793239711250 n_{20}
+8833453364736 n_{21}
+3557221631964 n_{22}
\\ & \qquad
+865778809896 n_{23}
+96197645544 n_{24}
,\\
\lambda(3,1,1)&= 
210 n_{8}
+45360 n_{9}
+642600 n_{10}
+3326400 n_{11}
+8523900 n_{12}
+12132120 n_{13}
\\ & \qquad
+9900891 n_{14}
+4414410 n_{15}
+840840 n_{16}
,\\
\lambda(2,2,1)& = 
280 n_{8}
+65520 n_{9}
+919800 n_{10}
+4596900 n_{11}
+11226600 n_{12}
+15315300 n_{13}
\\ & \qquad
+12192180 n_{14}
+5360355 n_{15}
+1009008 n_{16}
,\\ 
\lambda(2,1,1,1) &= 
892080 n_{9}
+83349000 n_{10}
+2170822500 n_{11}
+26591796000 n_{12}
\\ & \qquad
+189359450280 n_{13}
+876055780600 n_{14}
+2806801697700 n_{15}
\\ & \qquad
+6458643391200 n_{16}
+10866964308200 n_{17}
+13416110908200 n_{18}
\\ & \qquad
+12029730132420 n_{19}
+7626284265600 n_{20}
+3240681948504 n_{21}
\\ & \qquad
+828136252944 n_{22}
+96197645544 n_{23}
,\\ 
\lambda(1,1,1,1,1)&= 
1829520 n_{9}
+817016760 n_{10}
+72235270800 n_{11}
+2647690791900 n_{12}
\\ & \qquad
+53345363951880 n_{13}
+682682216596380 n_{14}
+6039039035429400 n_{15}
\\ & \qquad
+38946366176117400 n_{16}
+189638773413289200 n_{17}
+713826716560797840 n_{18}
\\ & \qquad
+2110340393930648880 n_{19}
+4950304696313776800 n_{20}
\\ & \qquad
+9265441477593100800 n_{21}
+13857900072549583050 n_{22}
\\ & \qquad
+16518003442667606880 n_{23}
+15574944975706176060 n_{24}
\\ & \qquad
+11462658924203487000 n_{25}
+6442333859931445476 n_{26}
\\ & \qquad
+2669265333214159680 n_{27}
+768162249080226000 n_{28}
\\ & \qquad
+137087416758932640 n_{29}
+11423951396577720 n_{30}
,\\
\lambda(4,1,1)& = 
7560 n_{9}
+113400 n_{10}
+589050 n_{11}
+1432200 n_{12}
+1747746 n_{13}
\\ & \qquad
+1009008 n_{14}
+210210 n_{15}
,\\
\lambda(3,2,1) &= 
30240 n_{9}
+453600 n_{10}
+2263800 n_{11}
+5128200 n_{12}
+5765760 n_{13}
\\ & \qquad
+3111108 n_{14}
+630630 n_{15}
,\\
\lambda(3,1,1,1)& = 
181440 n_{9}
+20594700 n_{10}
+585169200 n_{11}
+7466867100 n_{12}
+53238345160 n_{13}
\\ & \qquad
+236878922280 n_{14}
+700100200800 n_{15}
+1424183961200 n_{16}
\\ & \qquad
+2028644217600 n_{17}
+2026021217220 n_{18}
+1394893019520 n_{19}
\\ & \qquad
+633079366920 n_{20}
+171102531600 n_{21}
+20912531640 n_{22}
,\\
\lambda(2,2,2) &= 
7560 n_{9}
+113400 n_{10}
+554400 n_{11}
+1201200 n_{12}
+1261260 n_{13}
\\ & \qquad
+630630 n_{14}
+126126 n_{15}
,\\
\lambda(2,2,1,1) & = 
430920 n_{9}
+48365100 n_{10}
+1342768350 n_{11}
+16595271000 n_{12}
+114090786810 n_{13}
\\ & \qquad
+489169180500 n_{14}
+1396656261000 n_{15}
+2757820665600 n_{16}
\\ & \qquad
+3836191659300 n_{17}
+3764834613540 n_{18}
+2561038249770 n_{19}
\\ & \qquad
+1152905153400 n_{20}
+309695582196 n_{21}
+37642556952 n_{22}
,\\
\lambda(2,1,1,1,1) &= 
2872800 n_{9}
+1676581200 n_{10}
+164904790050 n_{11}
+6298213521600 n_{12}
\\ & \qquad
+126458717144760 n_{13}
+1557569654921280 n_{14}
+12897279885364875 n_{15}
\\ & \qquad
+76164200116804800 n_{16}
+333794628241774700 n_{17}
\\ & \qquad
+1115520582743887320 n_{18}
+2894974312598468100 n_{19}
\\ & \qquad
+5900420897320950000 n_{20}
+9496944246098058750 n_{21}
\\ & \qquad
+12073014477589665600 n_{22}
+12056810514853269165 n_{23}
\\ & \qquad
+9346203461860705440 n_{24}
+5507792588210012100 n_{25}
\\ & \qquad
+2383822869473725680 n_{26}
+714306478210645320 n_{27}
\\ & \qquad
+132360264456900480 n_{28}
+11423951396577720 n_{29}
,\\
\lambda(1,1,1,1,1,1) &= 
4011840 n_{9}
+11413776150 n_{10}
+3444031510920 n_{11}
+341035483477150 n_{12}
\\ & \qquad
+16334107213023600 n_{13}
+458689729433265330 n_{14}
\\ & \qquad
+8450977741650944500 n_{15}
+109792467460902806580 n_{16}
\\ & \qquad
+1056347419381332078000 n_{17}
+7792389750829016643310 n_{18}
\\ & \qquad
+45197004798213378970860 n_{19}
+209996223288982641611100 n_{20}
\\ & \qquad
+792475775069757320141600 n_{21}
+2453913578583257706865950 n_{22}
\\ & \qquad
+6280395970196377852122300 n_{23}
+13348940867374682005436000 n_{24}
\\ & \qquad
+23623379361553534532970000 n_{25}
+34820458479536167782093750 n_{26}
\\ & \qquad
+42668658867461724953856000 n_{27}
+43277385167426660997596850 n_{28}
\\ & \qquad
+36064655494545316433394600 n_{29}
+24416711436628708549852500 n_{30}
\\ & \qquad
+13209334741333731036156120 n_{31}
+5572094138384063443144992 n_{32}
\\ & \qquad
+1765284170332337557943040 n_{33}
+394963790210911659497760 n_{34}
\\ & \qquad
+55628702846607275985600 n_{35}
+3708580189773818399040 n_{36}
.\end{align*}

\begin{align*}
\chi(\AAA_{n,6},t)&=
t^6 -n_{6} t^5
+ \Big[
-n_{5}
+6 n_{6}
+210 n_{8}
+840 n_{9}
+1575 n_{10}
+1386 n_{11}
+462 n_{12}
\Big]t^4
\\
& \quad
+ 
\Big[
-n_{4}
+5 n_{5}
-15 n_{6}
-1750 n_{8}
-49420 n_{9}
-638190 n_{10}
-4174170 n_{11}
\\ & \qquad 
-16052344 n_{12}
-39399360 n_{13}
-63588525 n_{14}
-67267200 n_{15}
-44900856 n_{16}
\\ & \qquad 
-17153136 n_{17}
-2858856 n_{18}
\Big]t^3
\\ & \quad 
+
\Big[
-n_{3}
+4 n_{4}
-10 n_{5}
+20 n_{6}
+6545 n_{8}
+808269 n_{9}
+40056051 n_{10}
+907650744 n_{11}
\\ & \qquad 
+11350086517 n_{12}
+88888289490 n_{13}
+471662471280 n_{14}
+1779383330725 n_{15}
\\ & \qquad 
+4911968241180 n_{16}
+10078630954392 n_{17}
+15457237248453 n_{18}
\\ & \qquad 
+17651874149910 n_{19}
+14793239711250 n_{20}
+8833453364736 n_{21}
\\ & \qquad 
+3557221631964 n_{22}
+865778809896 n_{23}
+96197645544 n_{24}
\Big]t^2
\\ & \quad
+ \Big[
-n_{2}
+3 n_{3}
-6 n_{4}
+10 n_{5}
-15 n_{6}
-12992 n_{8}
-6922440 n_{9}
-1253314020 n_{10}
\\ & \qquad 
-85462425510 n_{11}
-2844981329190 n_{12}
-55072066969920 n_{13}
\\ & \qquad 
-692509356770231 n_{14}
-6077776702187190 n_{15}
-39056313885629040 n_{16}
\\ & \qquad 
-189868515062882656 n_{17}
-714183552466036653 n_{18}
-2110751767192248180 n_{19}
\\ & \qquad 
-4950652071117753000 n_{20}
-9265650239791905960 n_{21}
-13857984617732497242 n_{22}
\\ & \qquad 
-16518024125161398840 n_{23}
-15574947284449669116 n_{24}
-11462658924203487000 n_{25}
\\ & \qquad 
-6442333859931445476 n_{26}
-2669265333214159680 n_{27}
-768162249080226000 n_{28}
\\ & \qquad 
-137087416758932640 n_{29}
-11423951396577720 n_{30}
\Big]t
\\ & \quad
+
\Big[
-n
+2 n_{2}
-3 n_{3}
+4 n_{4}
-5 n_{5}
+6 n_{6}
+13020 n_{8}
+30306276 n_{9}
+21482580105 n_{10}
\\ & \qquad 
+4476460758924 n_{11}
+385724720114965 n_{12}
+17372731634141884 n_{13}
\\ & \qquad 
+473573588684378182 n_{14}
+8594435149519438219 n_{15}
+110777945174652868112 n_{16}
\\ & \qquad 
+1061369389494699244960 n_{17}
+7811915116061435525146 n_{18}
\\ & \qquad 
+45256066666226567391240 n_{19}
+210137040306764298507780 n_{20}
\\ & \qquad 
+792742452459639413305574 n_{21}
+2454315914651904858849294 n_{22}
\\ & \qquad 
+6280878741810399702094119 n_{23}
+13349398508879560267264124 n_{24}
\\ & \qquad 
+23623717675088602759587900 n_{25}
+34820649359673839255319726 n_{26}
\\ & \qquad 
+42668738231115243168001080 n_{27}
+43277408079933868947476370 n_{28}
\\ & \qquad 
+36064659595743867804796080 n_{29}
+24416711779347250447184100 n_{30}
\\ & \qquad 
+13209334741333731036156120 n_{31}
+5572094138384063443144992 n_{32}
\\ & \qquad 
+1765284170332337557943040 n_{33}
+394963790210911659497760 n_{34}
\\ & \qquad 
+55628702846607275985600 n_{35}
+3708580189773818399040 n_{36}
\Big]
.
\end{align*}

\begin{align*}
\mu_{n,6}(\set{\emptyset})&=
-1
+n
-n_{2}
+n_{3}
-n_{4}
+n_{5}
-n_{6}
-5033 n_{8}
-24143525 n_{9}
-20268685521 n_{10}
\\ & \qquad
-4391901811374 n_{11}
-382891072820410 n_{12}
-17317748416062094 n_{13}
\\ & \qquad
-472881550926490706 n_{14}
-8588359152133314554 n_{15}
\\ & \qquad
-110738893772690579396 n_{16}
-1061179531058250163560 n_{17}
\\ & \qquad
-7811200947966203878090 n_{18}
-45253955932111249292970 n_{19}
\\ & \qquad
-210132089669486420466030 n_{20}
-792733186818233074764350 n_{21}
\\ & \qquad
-2454302056670844347984016 n_{22}
-6280862223787140319505175 n_{23}
\\ & \qquad
-13349382933932372015240552 n_{24}
-23623706212429678556100900 n_{25}
\\ & \qquad
-34820642917339979323874250 n_{26}
-42668735561849909953841400 n_{27}
\\ & \qquad
-43277407311771619867250370 n_{28}
-36064659458656451045863440 n_{29}
\\ & \qquad
-24416711767923299050606380 n_{30}
-13209334741333731036156120 n_{31}
\\ & \qquad
-5572094138384063443144992 n_{32}
-1765284170332337557943040 n_{33}
\\ & \qquad
-394963790210911659497760 n_{34}
-55628702846607275985600 n_{35}
\\ & \qquad
-3708580189773818399040 n_{36}
.
\end{align*}

\normalsize
\bibliographystyle{plainnat}
\bibliography{x2}

\begin{thebibliography}{7}
\providecommand{\natexlab}[1]{#1}
\providecommand{\url}[1]{\texttt{#1}}
\expandafter\ifx\csname urlstyle\endcsname\relax
  \providecommand{\doi}[1]{doi: #1}\else
  \providecommand{\doi}{doi: \begingroup \urlstyle{rm}\Url}\fi

\bibitem[Athanasiadis(1999)]{athanasiadis2000}
Christos~A. Athanasiadis.
\newblock The largest intersection lattice of a discriminantal arrangement.
\newblock \emph{Beitr\"age Algebra Geom.}, 40\penalty0 (2):\penalty0 283--289,
  1999.
\newblock ISSN 0138-4821.

\bibitem[Bayer and Brandt(1997)]{bayer-brandt1997}
Margaret~M. Bayer and Keith~A. Brandt.
\newblock Discriminantal arrangements, fiber polytopes and formality.
\newblock \emph{J. Algebraic Combin.}, 6\penalty0 (3):\penalty0 229--246, 1997.
\newblock ISSN 0925-9899.
\newblock \doi{10.1023/A:1008601810383}.

\bibitem[Falk(1994)]{falk1994}
Michael Falk.
\newblock A note on discriminantal arrangements.
\newblock \emph{Proc. Amer. Math. Soc.}, 122\penalty0 (4):\penalty0 1221--1227,
  1994.
\newblock ISSN 0002-9939.
\newblock \doi{10.2307/2161193}.

\bibitem[Manin and Schechtman(1989)]{manin-schechtman1989}
Yu.~I. Manin and V.~V. Schechtman.
\newblock Arrangements of hyperplanes, higher braid groups and higher {B}ruhat
  orders.
\newblock In \emph{Algebraic number theory}, volume~17 of \emph{Adv. Stud. Pure
  Math.}, pages 289--308. Academic Press, Boston, MA, 1989.

\bibitem[Orlik and Terao(1992)]{orlik-terao-book}
Peter Orlik and Hiroaki Terao.
\newblock \emph{Arrangements of hyperplanes}, volume 300 of \emph{Grundlehren
  der Mathematischen Wissenschaften [Fundamental Principles of Mathematical
  Sciences]}.
\newblock Springer-Verlag, Berlin, 1992.
\newblock ISBN 3-540-55259-6.

\bibitem[Stanley(1997)]{stanley-book-1}
Richard~P. Stanley.
\newblock \emph{Enumerative combinatorics. {V}ol. 1}, volume~49 of
  \emph{Cambridge Studies in Advanced Mathematics}.
\newblock Cambridge University Press, Cambridge, 1997.
\newblock ISBN 0-521-55309-1; 0-521-66351-2.
\newblock With a foreword by Gian-Carlo Rota, Corrected reprint of the 1986
  original.

\bibitem[Stanley(2007)]{stanley-introduction}
Richard~P. Stanley.
\newblock An introduction to hyperplane arrangements.
\newblock In \emph{Geometric combinatorics}, volume~13 of \emph{IAS/Park City
  Math. Ser.}, pages 389--496. Amer. Math. Soc., Providence, RI, 2007.

\end{thebibliography}
\end{document}